\documentclass[a4paper,12pt]{amsart}

\pdfoutput=1
\usepackage{setspace}
\usepackage[left=22mm,head=30mm,bottom=30mm,foot=20mm,right=22mm]{geometry}
\usepackage{amsmath,amsthm,amsfonts,amssymb,mathtools,mathrsfs,url,bm,enumitem,stackengine}
\usepackage{newtxtext}
\usepackage[varvw]{newtxmath}
\allowdisplaybreaks
\usepackage{hyperref}
\hypersetup{
	colorlinks=true,
	linkcolor=blue,
	citecolor=red,%magenta,
}
\usepackage{appendix}

\mathtoolsset{showonlyrefs=true}
\theoremstyle{definition}

\newtheorem{remark}{Remark}[section]

\newtheorem{problem}{Problem}[section]
\newtheorem{assumption}{Assumption}[section]

\theoremstyle{plain}

\newtheorem{theorem}{Theorem}[section]

\newtheorem{corollary}{Corollary}[section]

\newtheorem{lemma}{Lemma}[section]
\def\R{\mathbb{R}}

\renewcommand{\phi}{\varphi}
\newcommand\restr[2]{{% we make the whole thing an ordinary symbol
		\left.\kern-\nulldelimiterspace % automatically resize the bar with \right
		#1 % the function
		\vphantom{\big|} % pretend it's a little taller at normal size
		\right|_{#2} % this is the delimiter
}}

\let\phi\varphi
\let\epsilon\varepsilon

\usepackage{tikz-cd}
\tikzset{
	symbol/.style={
		draw=none,
		every to/.append style={
			edge node={node [sloped, allow upside down, auto=false]{$#1$}}}
	}
}

\def\R{\mathbb R}

\def\proscal3#1#2{<\!\!#1, #2\!\!>_{_{{\hskip-4pt\R^{3}}}}}

\def\vu{ u }

\def\l2{L^2(\R^{n})}
\def\L2{L^2(\R^{2n})}

\def\mat22#1#2#3#4{\begin{pmatrix}#1&#2\\ #3&#4\end{pmatrix}}

\begin{document}

	\title[Remarks on a Liouville-type theorem by Chae and Wolf for SNS]{Remarks on a Liouville-type theorem by Chae and Wolf for stationary Navier-Stokes equations}
	
	\author{Raúl Fierro}
\address{Instituto de Matemáticas, Universidad de Valparaíso, Valparaíso, Chile.}
\email{raul.fierro@uv.cl}	

	\author{Gast\'{o}n Vergara-Hermosilla}
	\address{School of Mathematics, Harbin Institute of Technology, Xidazhi Street, Harbin 150001,
Heilongjiang, People’s Republic of China}
\email{gaston.v-h@outlook.com (corresponding author)}

	\maketitle
	
\begin{abstract}
In this note, we revisit a Liouville-type theorem of Chae and Wolf for stationary Navier--Stokes equations in $\mathbb{R}^3$ [J. Differential Equations 261 (2016) 5541–5560]. We show that their logarithmic improvement of the classical $L^{9/2}$ condition is part of a substantially broader weighted framework. More precisely, we prove that a solution $u\in\dot H^1(\mathbb{R}^3)$ is necessarily trivial whenever
\[
\int_{\mathbb{R}^3}|u(x)|^{9/2}\,\omega(|u(x)|)\,dx<
+\infty,
\]
for every positive, nondecreasing and bounded weight $\omega$ satisfying a mild  growth condition near the origin. 
This structural condition encompasses the logarithmic weight due to Chae and Wolf, as well as a hierarchy of iterated-logarithmic weights and (genuinely) non-logarithmic examples,  including a dyadic weight.
 Our result identifies a broader class of weighted integrability conditions under which the triviality of stationary Navier--Stokes solutions follows, and shows that the mechanism underlying the Chae--Wolf improvement is not intrinsically tied 
 to a specific logarithmic weight or to a single logarithmic scale.
 
\end{abstract}

\medskip

\textit{Keywords}: Stationary  Navier-Stokes equations; Liouville-type theorems; Weights.

\medskip
	
\textit{Mathematics Subject Classification (2020)}: 35Q30, 35B53, 76D05.

\tableofcontents
%---------------------------------------------------
 \section{Introduction and main results}
 
 	In this note  we study the stationary Navier--Stokes equations in the whole space  $\mathbb{R}^3$:
\begin{equation}\label{SNS}
-\Delta u + u \cdot \nabla u + \nabla P = 0, \quad 
\nabla \cdot u = 0,
\end{equation}
where $u:\R^3 \to \R^3$ denotes the velocity field and $P:\R^3 \to \R$ stands for the associated pressure. 
It is well known that solutions $(u,P)$ to \eqref{SNS} can be constructed in the spaces $ \dot{H}^1(\mathbb{R}^3) \times \dot{H}^{1/2}(\mathbb{R}^3)$ (see  \cite[Theorem 16.2]{lemarie2016navier}). However, uniqueness in this class remains an open (and difficult) question. 
This motivates the following problem, originally raised in \cite[Remark X.9.4]{galdi2011introduction} and \cite{Ser2016}.

\begin{problem}\label{Conjecture1}
Prove that any solution $u$ of \eqref{SNS} satisfying
\begin{equation}\label{Conjecture}
u \in \dot{H}^1(\mathbb{R}^3)
\qquad \text{and} \qquad
\lim_{ |x| \to +\infty}
u(x) = 0 ,
\end{equation}
is identically equal to zero.
\end{problem}
We now briefly summarize some of the key advances made on this problem. The Sobolev embedding theorem guarantees that every $u \in \dot{H}^1(\mathbb{R}^3)$ lies in $L^6(\mathbb{R}^3)$, which already provides a certain rate of decay at infinity. Nevertheless, this fact by itself does not appear to be enough to conclude that the solution must vanish. A number of partial results addressing Problem \ref{Conjecture1} have accumulated over time, each fo them show extra structural hypotheses forcing $u \equiv 0$. 
One of the first contributions in this direction is due to G.~Galdi \cite{galdi2011introduction}, who established that the condition $u \in L^{9/2}(\mathbb{R}^3)$ is sufficient to guarantee $u \equiv 0$. In what follows, we will refer to this result as Galdi's condition.
 Recently, N.~Lerner \cite{Lerner26} observed that this global $L^{9/2}$ hypothesis can be weakened by separating the low- and high-frequency parts of $u$. He in fact demonstrated   that it suffices to assume $u_{[0]} \in L^{9/2}(\mathbb{R}^3)$, where $u_{[0]}$ denotes the component of $u$ obtained by projecting onto vector fields whose Fourier transform is supported in a neighborhood of the origin. 
More recently, the second author of this paper showed \cite{V2026} that the uniqueness of the trivial solution follows if $u \in L^{9/2 + \varepsilon(\cdot)}(\mathbb{R}^3)$, where $\varepsilon(\cdot)>0$. 
Following a different strategy, H. Kozono, Y. Terasawa, and Y. Wakasugi  established in \cite{Kozonoetal} that whenever the weak-$L^{9/2}$ norm of $u$ obeys the bound
\[
\|u\|_{L^{9/2,\infty}} \leq \delta \bigl( \nu \| \operatorname{curl} u \|_{L^2}^2 \bigr)^{1/3},
\]
with $\delta$ small enough,  in fact  yield that $u \equiv 0$. On the other hand,  G. Seregin and W. Wang  subsequently generalized this result in \cite{Sereginetwang}. 
A further idea to attack the problem imposes conditions directly on the Laplacian of $u$. D.~Chae \cite{chae14} proved that $\Delta u \in L^{6/5}(\mathbb{R}^3)$ in fact implies that $u \equiv 0$.  In addition to these strategias, has been stressed that some structural hypotheses provide uniqueness of the trivial solution. In \cite{Seregin16} G. Seregin  demonstrated that if $u$ can be written as $u = \operatorname{curl} w$ for some $w \in \mathrm{BMO}(\mathbb{R}^3)$, then $u$ necessarily vanishes.\\

Notably, Galdi's condition was later relaxed by D.~Chae and J.~Wolf \cite{ChaeWolf}, who proved that the weaker assumption
\begin{equation}\label{chae-wolf.assumption}
\int_{\mathbb{R}^3} |u(x)|^{9/2} \bigl[\log(2 + |u(x)|^{-1})\bigr]^{-1} \, dx < +\infty
\end{equation} 
still guarantees $u \equiv 0$, providing a logarithmic improvement of Galdi’s result. 
Indeed, since  $ \bigl[\log(2 + |u(x)|^{-1})\bigr]^{-1} $ is bounded above  by $ \bigl[\log(2  )\bigr]^{-1} $, we can write   
\[
\int_{\mathbb{R}^3} |u(x)|^{9/2} \bigl[\log(2 + |u(x)|^{-1})\bigr]^{-1} dx  
\lesssim 
\int_{\mathbb{R}^3} |u(x)|^{9/2}   dx
. 
\]
Consequently, the assumption $u \in L^{9/2}(\mathbb{R}^3)$ implies \eqref{chae-wolf.assumption}. The converse, however, does not hold.\\

Motivated by the above result, in this note we seek to place the Chae--Wolf theorem within a more general framework of weighted integrability conditions leading to Liouville-type theorems for the stationary Navier--Stokes equations in $\mathbb{R}^3$. More precisely, we replace the specific logarithmic weight considered by Chae and Wolf with a general admissible weight and show that the resulting theorem yields, in particular, improvements of Galdi's (and Chae--Wolf's) results.\\

To this end, we begin by introducing the following assumptions. 

\begin{assumption}\label{ass:omega}
We assume that $\omega:(0, + \infty)\to(0, + \infty)$ is continuous and satisfies:
\begin{enumerate}
\item[(H1)] $\omega$ is nondecreasing,
\item[(H2)] $\omega$ is bounded: there exists $C_\omega>0$ such that $\omega(t)\le C_\omega$ for all $t>0$,
\item[(H3)] there exist constants $C>0$ and $t_0\in(0,1)$ such that
\[
\omega(t)^{-1} \le C\,\log(1/t) \qquad \text{for all } 0<t<t_0.
\]
\end{enumerate}
\end{assumption}
Under Assumption \ref{ass:omega}, we can now state our main theorem.
 
 %%%%%%%%%%%%%%%%%%%
\begin{theorem}\label{thm:main1} 
 Let $ v  \in \dot{H}^1(\mathbb{R}^3) $ be a solution of (\ref{SNS}) and $\omega$ be a weight function satisfying hypotheses (H1)--(H3) in Assumption  \ref{ass:omega}.  Suppose further that
 \[
 \int_{	\R^3	}
| v |^{\frac 9 2}
\omega(|v|)
dx
< + \infty .
 \]
 Then, $v $ is identically equal to zero.
\end{theorem}
%%%%%%%%%%%%%%%%%%%

 Some remarks are in order here.
\begin{remark}\label{rmk11}
The class of admissible weights in Theorem~\ref{thm:main1} contains, in particular, the logarithmic family

$$
\omega_\beta(t)
=
\left[\log\left(2+\frac1t\right)\right]^{-\beta},
\qquad t>0,
\qquad \beta\in(0,1].
$$
Consequently, for every $\beta\in(0,1]$, we can (roughly) write 
$$
\int_{\mathbb R^3}
|v|^{9/2}
\left[\log\left(2+\frac1{|v|}\right)\right]^{-\beta}
\,dx
<
+ \infty
\quad\Longrightarrow\quad
v\equiv0.
$$
Note in particular that the choice $\beta=1$ recovers precisely the assumption considered  by Chae and Wolf \cite{ChaeWolf}. 
To continue, note that the scope of the theorem, however, is not restricted to powers of a single logarithm. Indeed, for a fixed integer $k\geq1$, we can consider the iterated-logarithmic (or nested logarithms) weight

$$
\omega_k(t)
=
\frac{1}{
\log\Bigl(
2+\log\bigl(
2+\cdots+\log(2+\tfrac1t)\cdots
\bigr)
\Bigr)
},
\qquad t>0,
$$
with $k$ nested logarithms. These weights provide a natural hierarchy of admissible corrections and further illustrate the flexibility of the framework. \\

On the other hand, we remark that the result of Chae and Wolf should be viewed as one distinguished instance within a substantially broader class of weighted integrability conditions. Theorem~\ref{thm:main1} is formulated in terms of the structural properties {\rm(H1)--(H3)}, rather than a prescribed logarithmic form, thereby allowing for a variety of asymptotic behaviors of the weight near the origin. 

The corresponding logarithmic and iterated-logarithmic examples are discussed in detail in Subsections \ref{subsec.4.1} and \ref{subsec.4.2}, respectively.
\end{remark}

\begin{remark}[A non-logarithmic example]\label{rmk22}
The scope of Theorem~\ref{thm:main1} can be further illustrated by considering following  dyadic weight\footnote{for the precise definition of $N(t)$ please see Subsection \ref{subsec.4.4}.}
$$
\omega_{\mathrm{dyadic}}(t)
:=
\begin{cases}
1,
& t\geq1,
\\[6pt]
\displaystyle
\frac1{n+2}
+
\left(
\frac1{n+1}-\frac1{n+2}
\right)
\frac{t-2^{-(n+1)}}{2^{-n}-2^{-(n+1)}},
&
t\in\bigl(2^{-(n+1)},2^{-n}\bigr],
\quad n=N(t).
\end{cases}
$$
In contrast to the examples mentioned above, this weight is constructed through a dyadic interpolation and is not defined in terms of logarithmic functions. Nevertheless, it belongs to the class covered by Theorem~\ref{thm:main1}, and hence yields another Liouville-type theorem, which roughly reads as 
$$
\int_{\mathbb R^3}
|v|^{9/2}\omega_{\mathrm{dyadic}}(|v|)
\,dx
<
+ \infty
\quad\Longrightarrow\quad
v\equiv0.
$$
This example allow us to conclude that the admissible class (given by assumption 1.1) contains weights arising from fundamentally different constructions and is therefore genuinely broader than the logarithmic scale. In this sense, Theorem~\ref{thm:main1} provides a structural framework for Liouville-type criteria rather than a single logarithmic refinement of the classical $L^{9/2}$ condition.

This example is discussed in detail in Subsections \ref{subsec.4.4}.
\end{remark}

 % remark 3
\begin{remark}\label{rmk33}
Note that, the behavior of the weight near the origin determines whether the weighted condition in Theorem~\ref{thm:main1} provides a genuine improvement over Galdi's condition. More precisely, if
\[
\lim_{t\to0^+}\omega(t)=0,
\]
then the weight reduces the contribution of the region where $|v|$ is small, and the corresponding weighted integrability condition can be strictly weaker than
$v\in L^{9/2}(\mathbb R^3)$. This is precisely the mechanism behind the logarithmic improvement of Chae and Wolf and, more generally, behind the examples considered in the preceding remarks.

In contrast, if we have
\[ 
\lim_{t\to0^+}\omega(t)=\ell>0
,
\]
then no such improvement is possible. In fact,  by (H2) we have $\ell\le\omega(t)\le C_\omega$ for every $t>0$, then, we can write 
\[
\ell\int_{\mathbb R^3}|v|^{9/2}\,dx
\;\le\;
\int_{\mathbb R^3}|v|^{9/2}\omega(|v|)\,dx
\;\le\;
C_\omega\int_{\mathbb R^3}|v|^{9/2}\,dx,
\]
and we conclude that the weighted integral is finite if and only if $v\in L^{9/2}(\mathbb R^3)$, {\it i.e.}
the class of vector fields covered by Theorem~\ref{thm:main1} coincides, up to a
 constant, with Galdi's class, and no (genuine)
  improvement occurs.
  
Now, it is important to highlight that the assumptions {\rm(H1)--(H3)} also admit weights that do not yield any improvement over the result of Galdi. For instance, consider the weights
$$
\omega_1(t):=2+\arctan(t),
\quad
\omega_2(t):=3-\frac{1}{1+t},
\quad
\omega_3(t) =
\begin{cases} 
2 - e^{-t}, & 0 < t \leq \frac{1}{2}, \\ 
2 - e^{-1/2}, & t > \frac{1}{2},
\end{cases}
\qquad t>0.
$$
These weights satisfy (H1)--(H3). However, is easy to check that
\[
\lim_{t\to0^+}\omega_j(t)>0,
\qquad j=1,2,3,
\]
whereas the weights considered in the preceding remarks satisfy
$\lim_{t\to0^+}\omega(t)=0$. 
These new examples are discussed in detail in Subsections \ref{subsec.4.3} and \ref{subsec.4new}.
\end{remark}

 % remark 4
\begin{remark}\label{rem.4}
One of the main building blocks in the proof of our main result is the estimate 
\[
\int_{B_{R/2}} |\nabla v|^2 \, dx \lesssim 
R^{-1} \int_{B_R \setminus B_{R/2}} |v|^3 \, dx 
+ \left( R^{-\frac{1}{6}} + R^{-\frac{7}{6}} \int_{B_{2R}} |v|^3 \, dx \right),
\]
which is proved in detail in \cite{ChaeWolf}. 
The role of the generic weight $\omega$ enters to handle the term on the right-hand side of the estimate. In that sense, assumption (H3) plays a role for proving that the term $R^{-1} \int_{B_R \setminus B_{R/2}} |v|^3 \, dx $ goes to 0 as $R \to +\infty$.
\end{remark}

The remainder of this paper is organized as follows. In Section \ref{section.preliminaries}, we state some   results that will be needed later. Section \ref{section.proof.main.result} contains the proof of our main theorem. In Section \ref{section.examples}, we provide a number of concrete examples of weights $\omega$ satisfying assumptions (H1)--(H3), including logarithmic and dyadic weights. Finally, Section \ref{proof.main.lemma} is devoted to the proof of Lemma 2.1.\\

Given $R>1$, in what follows we consider the following notation
\[
A_R := \{ x \in \mathbb{R}^3 \ : \ R/2 < |x| < R \}.
\]

%%%%%%%%%%%%%%%%%%
\section{Preliminaries}\label{section.preliminaries}
%%%%%%%%%%%%%%%%%%

In this section we collect a remark and a lemma, which will play important roles in the proof of our main result.  
%%%%
%\begin{remark}\label{rmk17}
%Condition (H3) implies that for every $\alpha>0$,
%\[
%t^\alpha\,\omega(t)^{-1}\longrightarrow 0 \qquad \mbox{as }t\to 0^+.
%\]
%Indeed, for $0<t<t_0$, we know that $t^\alpha\omega(t)^{-1}\le C_3\,t^\alpha\log(1/t)$, and  
%$t^\alpha\log(1/t) \to0$ as $t\to0^+$, since a logarithm is dominated by any positive power.
%\end{remark}
%%%%
\begin{remark}\label{rkm17}
Condition (H3) implies that, for every $\alpha > 0$, there exists a constant $C_\alpha > 0$ such that
\[
t^\alpha \omega(t)^{-1} \leq C_\alpha,
\qquad 0 < t < 1.
\]
Indeed, for $0 < t < t_0$, the assumption (H3) yields
\[
t^\alpha \omega(t)^{-1}
\leq C\,t^\alpha \log(1/t).
\]
Now, since $ t^\alpha\log(1/t)\to 0$
as $t\to0^+$, we have that  the map  $t\mapsto t^\alpha\log(1/t)$ admits a continuous extension to $[0,t_0]$ and and is therefore bounded. Denote such a bound by $M_\alpha>0$. Thus, we can write 
\[
t^\alpha\omega(t)^{-1}\leq CM_\alpha,
\qquad 0<t<t_0.
\]
On the other hand, by (H1), we know that $\omega(t)\geq\omega(t_0)>0$,
for $t_0\leq t<1$, and hence we get 
\[
t^\alpha\omega(t)^{-1}
\leq \omega(t_0)^{-1},
\qquad t_0\leq t<1.
\]
Thus, by combining the two estimates and setting $C_\alpha:=\max\{CM_\alpha,\omega(t_0)^{-1}\}$, this proves the claim. In particular, note that 
\[
t^\alpha\omega(t)^{-1}\longrightarrow0
\qquad\text{as }t\to0^+.
\]
\end{remark}
%%%%

To continue, for $q>1$ we define
\[
\phi_q(t) := \int_0^t \xi^{q-1}\,\omega(\xi)\,d\xi, \qquad t\ge 0.
\]

\begin{lemma}\label{lem.phiq}
Under (H1)--(H2), $\phi_q$ is well defined, convex, and satisfies
\begin{enumerate}
\item $\phi_q(0)=0$;
\item $\phi_q \in C^1([0, + \infty))$, with $\phi_q'(t) = t^{q-1}\omega(t) > 0$ for $t>0$;
\item $\phi_q(t) \to +\infty$ as $t\to +\infty$;
\item $\displaystyle \phi_q(\tau) \le \frac{\tau^q\,\omega(\tau)}{q}$ for all $\tau\ge 0$.
\end{enumerate}
Consequently, $\phi_q$ admits an inverse $\phi_q^{-1}$ which is continuous, strictly increasing, with $\phi_q^{-1}(0)=0$.
\end{lemma}

The proof of this lemma will be given in the last section of the paper.
 
 \section{Proof of main result}\label{section.proof.main.result}

To begin, we stress the fact that by assumption we know that 
\begin{equation}\label{hipotesis}
\int_{\mathbb{R}^3} |v|^{\frac 9 2} \omega(|v|) \, dx < +\infty,
\end{equation}
where $\omega$ fulfills the assumptions (H1)--(H3).\\

Now note that, by the Sobolev embedding theorem, since $v\in \dot H^1(\mathbb{R}^3)$, we have $v\in L^6(\mathbb{R}^3)$, and hence, by the local embedding $L^6_{\mathrm{loc}}(\mathbb{R}^3)\hookrightarrow L^3_{\mathrm{loc}}(\mathbb{R}^3)$ 
we get that $ v\in L^3_{\mathrm{loc}}(\mathbb{R}^3)$.
Therefore, by \cite[Theorem X.1.1]{galdi2011introduction}, the velocity field $v$  and the associated pressure $P$ are smooth.

To continue, let  $\zeta \in C_c^\infty(B_R)$ be a cut-off function such that $0 \leq \zeta \leq 1$ in $B_R$, $\zeta \equiv 1$ on $B_{R/2}$, and  
\[
|\nabla \zeta| \lesssim R^{-1}, \quad |\nabla^2 \zeta| \lesssim R^{-2}
.
\]  
Thus, by testing the stationary Navier-Stokes equations with $\zeta v$, we get 
\[
\int_{\mathbb{R}^3} |\nabla v|^2 \zeta \, dx = \frac{1}{2} \int_{\mathbb{R}^3} |v|^2 \Delta \zeta \, dx + \frac{1}{2} \int_{\mathbb{R}^3} |v|^4 v \cdot \nabla \zeta \, dx + \int_{\mathbb{R}^3} p v \cdot \nabla \zeta \, dx.
\]
With this identity at hand, considering Hölder's and Young's inequalities, and the approach in \cite{galdi2011introduction} for handling the pressure term, we can write\footnote{As mentioned in Remark \ref{rem.4}, the details of the derivation of this estimate are given in \cite{ChaeWolf}, see pp.~5548--5549. We omit the derivation here for brevity.} 
\begin{equation}\label{estimate.base}
\int_{B_{R/2}} |\nabla v|^2 \, dx \lesssim R^{-1} \int_{A_R} |v|^3 \, dx + \left( R^{-\frac{1}{6}} + R^{-\frac{7}{6}} \int_{B_{2R}} |v|^3 \, dx \right)
=:
I_1 + I_2
.
\end{equation}
Our aim in the following is to prove that 
\[
(I_1,I_2) \to (0,0) \quad \text{as } R\to +\infty.
\]
 \subsection{Control of $I_2$.}
To begin we prove the following result.
\begin{lemma}
If $v\in \dot H^1$ and \eqref{hipotesis} follow, then $v\in L^p (\mathbb{R}^3)$, for $p\in (\tfrac 9 2, 6]$.
\end{lemma}
 \begin{proof}
First we split the domain $\mathbb{R}^3$ to write 
\[
\int_{\mathbb{R}^3} |v|^q = \int_{\{ |v| \leq 1 \}} |v|^q + \int_{\{ |v| > 1 \}} |v|^q = J_1 + J_2.
\]
First we control $J_1$. For doing it, we start by
considering $t=|v|$, and
\[
\alpha = q - \frac{9}{2} > 0.
\]
 By Remark \ref{rkm17} we know that there exists a numerical constant $C_q>0$ such that 
\[
t^{q-9/2} \omega(t) ^{-1}  \leq C_q
 \quad (0 < t   \leq 1).
\]
Hence, we can write 
\[
t^q \leq C_q  t^{9/2} \omega (t)  \quad (0 < t    \leq 1).
\]
and then 
\[
\int_{\{|v| \leq 1\}} |v|^q \, dx 
\lesssim
\int_{\{|v| \leq 1\}} |v|^{9/2} \omega(|v|) \, dx
\lesssim 
\int_{\mathbb{R}^3} |v|^{9/2} \omega(|v|)\, dx
<
\infty. 
\]
Now we control $J_2$.  {\color{black} Note that, the boundedness of  $\omega (t)$}  implies
\[
\int_{\{|v| > 1\}} |v|^{9/2} \omega(|v|) \, dx 
\lesssim
\int_{\{|v| > 1\}} |v|^{9/2}   \, dx
\lesssim
\int_{\{|v| > 1\}} |v|^{6}   \, dx
\lesssim 
\int_{\mathbb{R}^3} |v|^{6} \, dx
<\infty .
\]
With this we conclude the proof of this lemma.

\end{proof}
 
With this result at hand we proceed to control $I_2$.
To this end, we first write 
\[
\int_{B_{2R}} |v|^3 \, dx = |B_{2R}| \frac{1}{|B_{2R}|} \int_{B_{2R}} |v|^3 \, dx.
\]
Thus, since the map $
t \mapsto t^{3/q}
$
is concave (as $3/q < 1$), Jensen's inequality yield
\[
\frac{1}{|B_{2R}|} \int_{B_{2R}} |v|^3 \, dx \leq \left( \frac{1}{|B_{2R}|} \int_{B_{2R}} |v|^q \, dx \right)^{3/q}.
\]
Thus, we can write 
\[
\int_{B_{2R}} |v|^3 \, dx \leq |B_{2R}|^{1-\frac{3}{q}} \left( \int_{B_{2R}} |v|^q \, dx \right)^{3/q}.
\]
Then, considering that $|B_{2R}| \sim R^3,$ we get
\[
\int_{B_{2R}} |v|^3 \, dx \lesssim R^{3-\frac{9}{q}} \|v\|_{L^q(\mathbb{R}^3)}^{3}.
\]
Hence, we obtain 
\[
R^{-7/6} \int_{B_{2R}} |v|^3 \, dx \leq R^{\frac{11}{6} - \frac{9}{q}} \|v\|_{L^q}^3. \tag{2}
\]
Thus, if 
$\frac{9}{2} < q < \frac{54}{11},
$  we conclude
$
\frac{11}{6} - \frac{9}{q} < 0,
$ and then $
R^{\frac{11}{6} - \frac{9}{q}} \to 0.
$
Provided with this information we get
\[
R^{-7/6} \int_{B_{2R}} |v|^3 \, dx \to 0  \qquad \text{as }R\to +\infty
\]
and thus
\[
R^{-1/6} + R^{-7/6} \int_{B_{2R}} |v|^3 \, dx \to 0
\qquad \text{as }R\to +\infty
.
\]

 \subsection{Control of $I_1$.}

Considering
the   constant $\mathcal{C}= \frac 8 7 \left( \frac 4 3 \pi  \right)^{-1} $ and the fact that $\mathrm{meas}(A_R)\sim R^3$, we write 
\[
\phi_{\frac{3}{2}}\left(\frac{\mathcal{C}}{R}\int_{A_R}|v|^3\,dx\right)
=
 \phi_{\frac 32}\left(\frac{1}{ \mathrm{meas}(A_R) }\int_{A_R} R^2|v|^3\,dx\right).
\]
Then, since $\phi_{3/2}$ is convex, by considering Jensen's inequality and part (4) of Lemma \ref{lem.phiq}  we get 
 \begin{align*}
 \phi_{\frac 32}\left(\frac{\mathcal{C}}{R^3}\int_{A_R} R^2|v|^3\,dx\right)
& 
\le \frac{C}{R^3}\int_{A_R}\phi_{\frac 32}(R^2|v|^3)\,dx
\\
&\lesssim \int_{A_R} |v|^{\frac 92}\,\omega(R^2|v|^3)\,dx.
\end{align*}
Before continue, we state the following result due to Chae and Wolf \cite{ChaeWolf}. 
\begin{lemma}
Let $f\in L^1(\mathbb{R}^3)$. Then for every $\varepsilon>0$ there exists $R>\varepsilon^{-1}$ such that
\[
\int_{A_R} |f|\,dx \le \frac{\varepsilon}{\log R}.
\]
\end{lemma}
This Lemma has the following   consequence.
\begin{corollary}\label{cor.log}
For every $\varepsilon>0$ there exists $R>\varepsilon^{-1}$ such that
\[
\int_{A_R} |v|^{\frac 92}\,\omega(|v|)\,dx \le \frac{\varepsilon}{\log R}.
\]
\end{corollary}
To continue, given $\varepsilon \in (0,1)$, we decompose
\[
A_R = A_1 \cup A_2 := \left\{ R^2|v|^3 > \varepsilon \right\} \cup \left\{ R^2|v|^3 \le \varepsilon \right\},
\]
with the aim to estimate 
\[
\int_{A_2} |v|^{\frac92}\,\omega(R^2|v|^3)\,dx 
\quad \text{and} \quad 
 \int_{A_1} |v|^{\frac92}\,\omega(R^2|v|^3)\,dx
 .
\]

\begin{itemize}
 \item {\bf Estimate on $A_2$.}  By (H1), we know that  $\omega(R^2|v|^3)\le\omega(\varepsilon)$ on $A_2$, and then 
 \[
\int_{A_2} |v|^{\frac92}\,\omega(R^2|v|^3)\,dx
\le \omega(\varepsilon)\int_{A_2} |v|^{\frac 92}\,dx.
\]
Now,  considering that $|v|^{9/2}\le \varepsilon^{3/2}R^{-3}$ on $A_2$, together with the boundedness of $\omega$, we get  
\[
\int_{A_2} |v|^{\frac92}\,\omega(R^2|v|^3)\,dx
\le \frac{\varepsilon^{3/2}}{R^3}\,\omega(\varepsilon)\,\mathrm{meas}(A_2)
\lesssim \varepsilon^{3/2}\,\omega(\varepsilon) \lesssim \varepsilon^{3/2}.
\]

\item { \bf Estimate on $A_1$.} Note that,  by the  boundedness of $\omega$ in (H2), we can write
\[
\int_{A_1} |v|^{\frac 92}\,\omega(R^2|v|^3)\,dx \lesssim \int_{A_1} |v|^{\frac 92}\,dx.
\]
To continue, we emphasize the fact that on $A_1$ we can write 
\[
|v| > \left(\varepsilon R^{-2}\right)^{1/3} =: \tau_R.
\]
{\color{black} 
Note that, since $\omega$ is nondecreasing (by (H1)), we have  $\omega(|v|)^{-1} \le \omega(\tau_R)^{-1}$.
Moreover, considering that $R\ge \varepsilon^{-1}$, we have  $\tau_R=(\varepsilon R^{-2})^{1/3}\le(\varepsilon\cdot\varepsilon^2)^{1/3}=\varepsilon$. Considering this fact, and as we eventually take $\varepsilon\to0^+$, we may assume without loss of generality that $\varepsilon<t_0$, so that $\tau_R<t_0$ and (H3) applies, then we can write 
\[
\omega(\tau_R)^{-1} \lesssim \log(1/\tau_R).
\]
Now,  using that $\log(1/\tau_R) = \frac13\log\!\left(\frac{R^2}{\varepsilon}\right)$ and 
\[
\frac13\log\!\left(\frac{R^2}{\varepsilon}\right) \lesssim 
\log R, 
\]
we conclude  
\[
\omega(|v|)^{-1} \le 
\omega(\tau_R)^{-1} \lesssim \log R.
\]
on $A_1$.
}
Hence, we get
\[
\int_{A_1} |v|^{\frac92}\,dx = \int_{A_1} |v|^{\frac 92}\,\omega(|v|)\cdot\omega(|v|)^{-1}\,dx
\lesssim
 \log R \int_{A_1} |v|^{\frac92}\,\omega(|v|)\,dx,
\]
and by Corollary \ref{cor.log} we obtain
\[
\int_{A_1} |v|^{\frac92}\,\omega(R^2|v|^3)\,dx
\lesssim \log R \int_{A_1} |v|^{\frac92}\,\omega(|v|)\,dx
\lesssim \log R\cdot \frac{\varepsilon}{\log R} \lesssim \varepsilon.
\]
\end{itemize}
To continue, note that, by  combining the estimates on $A_1$ and $A_2$, we can write
\[
\phi_{\frac 32}\left(\frac{C}{R}\int_{A_R} |v|^3\,dx\right) \lesssim \varepsilon 
\quad (\mbox{for }0<\epsilon <1).
\]
Then, by Lemma \ref{lem.phiq}, we know that $\phi_{3/2}^{-1}$ exists, is continuous, increasing, and $\phi_{3/2}^{-1}(0)=0$. Applying $\phi_{3/2}^{-1}$ (which preserves the inequality ) we get
\[
\frac{C}{R}\int_{A_R} |v|^3\,dx \le \phi_{\frac 32}^{-1}(C\varepsilon),
\]
and since $\phi_{3/2}^{-1}(C\varepsilon)\to \phi_{3/2}^{-1}(0)=0$ as $\varepsilon\to 0$ ( or equivalently $R\to+ \infty$, since $R>\varepsilon^{-1}$), we conclude
\[
I_1=R^{-1}\int_{A_R} |v|^3\,dx \longrightarrow 0 \qquad \text{as } R \to +\infty.
\]

Thus, by mixing the limits for $I_{1}$ and $I_{2}$ with the estimate 
\eqref{estimate.base}, we conclude
\begin{equation}\label{LimiteSobolev}
\lim_{R\to +\infty}\int_{B_{R/2} }|\nabla u|^2 dx=\| u\|_{\dot{H}^1}=0.
\end{equation}
Then, by considering Sobolev embeddings, we get $\|\vu\|_{L^6}=0$, and in consequence  $\vu= 0$. 

\section{Examples}\label{section.examples}

In the following we prove that the examples of weight functions mentioned in Remarks \ref{rmk11}, \ref{rmk22} and \ref{rmk33}  in fact satisfy the assumptions (H1), (H2) and (H3).
\subsection{Example 1: $\omega_\beta(t) := \left[ \log\left(2 + \frac{1}{t}\right) \right]^{-\beta}$}\label{subsec.4.1}
For $\beta \in (0, 1]$, we define
\[
\omega_\beta(t) := \left[ \log\left(2 + \frac{1}{t}\right) \right]^{-\beta}.
\]
 
Consider $g(t) := \log(2 + \frac{1}{t})$, which we already know is a decreasing and positive function (since $g(t) > \log 2$). On the other hand, since $\beta > 0$, the map $g \mapsto g^{-\beta}$ is decreasing in $g$, and thus $\omega_\beta(t) = g(t)^{-\beta}$ is nondecreasing in the variable $t$.\\

On the other hand, as $t$ evolves to $+\infty$, we have that $g(t) \to \log 2$,  and therefore  $\omega_\beta(t) \to (\log 2)^{-\beta}$. Then, since $g(t) > \log 2$, we have $\omega_\beta(t) < (\log 2)^{-\beta} =: C_\omega$ for all $t>0$.\\

Finally, note that,  $ \omega_\beta(t)^{-1} = g(t)^\beta$,  $g(t) = \log(1/t) + \log(1+2t)$, and  for  $0<t\le1$, we have $\log(1+2t)\le\log3$. Thus, we get  $g(t)\le\log(1/t)+\log3$.
Now, since  $\log(1/t)\ge1$ for $t\le1/e$, we can write 
\[
g(t) \le (1+\log3)\log(1/t),
\]
for $0<t  \le \frac 1 e$. 
Then, considering that  $\beta\le1$, $x\mapsto x^\beta$ is an increasing map, and $[\log(1/t)]^\beta\le\log(1/t)$ for $t\le1/e$ (remember that $\log(1/t)\ge1$ in such interval), we conclude  
\[
\omega_\beta(t)^{-1} = g(t)^\beta \le (1+\log3)^\beta\,\log(1/t) \qquad \text{for all } 0<t<\tfrac1e.
\]
Thus (H3) holds considering $C=(1+\log3)^\beta$ and $t_0=1/e$.

%%%%%%
\subsection{Example 2: $k$ nested logarithms}\label{subsec.4.2}
For $k\ge 1$ an integer, we define
\[
L_0(t):=\frac1t,
\qquad
L_j(t):=\log\bigl(2+L_{j-1}(t)\bigr),
\quad j=1,\dots,k,
\qquad t>0,
\]
and we set
\[
\omega_k(t):=\frac1{L_k(t)}.
\]
Note that we can recast this equivalently as  the expression
\[
\omega_k(t)
=
\frac{1}{
\log\Bigl(2+\log\bigl(2+\cdots+\log(2+\tfrac1t)\cdots\bigr)\Bigr)
},
\]
with $k$ nested logarithms. We claim that $\omega_k$ satisfies (H1)--(H3) for every $k\ge 1$.\\

In the next, we will verify the three hypotheses separately.

\emph{Assumption (H1).}
Note that, for each $j\ge 0$, the functions $L_j$ are positives, strictly decreasing on $(0,+\infty)$, and
\[
\lim_{t\to 0^+}L_j(t)=+\infty.
\]
In fact, we can prove it by induction on $k$. f
For $j=0$, this follows from $L_0(t)=1/t$. Now we suppose that the claim holds for $L_{j-1}$. Thus, since $L_{j-1}(t)>0$, we can write 
\[
L_j(t)=\log\bigl(2+L_{j-1}(t)\bigr)\ge \log 2>0.
\]
Moreover,  as the map $x\mapsto \log(2+x)$ is strictly increasing, while $L_{j-1}$ is strictly decreasing, we get that $L_j$ is strictly decreasing. Finally, considering that  $L_{j-1}(t) \to + \infty$ as $t\to 0^+$, we get
\[
L_j(t)=\log\bigl(2+L_{j-1}(t)\bigr)\to + \infty.
\]
In particular, $L_k$ is positive and strictly decreasing, therefore we get that 
\[
\omega_k(t)=\frac1{L_k(t)}
\]
is strictly increasing, and thus  (H1) holds.

\emph{Assumption (H2).}
To begin, we define recursively
\[
c_0:=0,
\qquad
c_j:=\log(2+c_{j-1}),
\quad j\ge 1.
\]
Then, we easily get that $c_j>0$ for every $j\ge 1$.
In the following we will prove that 
\[
\lim_{t\to+\infty}L_j(t)=c_j
\]
for each $j\ge 0$. 
Our strategy will be by induction also. Indeed, this is immediate for $j=0$, since $L_0(t)\to 0=c_0$. 
Now, if we suppose that $L_{j-1}(t)\to c_{j-1}$, then, by continuity of  the map $x\mapsto \log(2+x)$, we can write 
\[
L_j(t)
=
\log\bigl(2+L_{j-1}(t)\bigr)
\longrightarrow
\log(2+c_{j-1})
=
c_j.
\]
Thus, since $L_k$ is strictly decreasing and converges to $c_k>0$ as $t\to +\infty$, we have for each $t>0$ that 
\[
L_k(t)>c_k
.
\]
Then, consequently, we get
\[
0<\omega_k(t)
=
\frac1{L_k(t)}
<
\frac1{c_k}
=:C_{\omega,k}.
\]
Hence, the assumption  (H2) holds.

\emph{Assumption (H3).}
In what follow we will use the following fact, for each $x\ge 2$,
we have\footnote{
This can be proven by setting
\[
\Delta(x):=x-\log(2+x).
\]
and stressing the fact that, for  $x>-1$ we have
\[
\Delta'(x)
=
1-\frac1{2+x}
=
\frac{1+x}{2+x}>0
.
\]
Now, note that
\[
\Delta(2)=2-\log 4>0.
\]
Hence, we get that $\Delta(x)>0$ for all $x\ge 2$ and the claim follows.
}
\begin{equation} \label{ineq log 2+x}
\log(2+x)\le x.
\end{equation}
To continue, we consider $k\ge 1$ fixed. Then, consiring the information proven above we know that,  for each $j=0,\dots,k-1$, we can write 
\[
L_j(t)\to +\infty
\qquad \text{ad }t\to 0^+
.
\]
Thus, for each $j=0,\dots,k-1$, there exists $s_j>0$ such that, for $0<t<s_j$, we get
\[
L_j(t)\ge 2
.
\]
To continue, we set $t_k^*:=\min_{0\le j\le k-1} s_j>0$.
Then, for $0<t<t_k^*$, we have $L_j(t)\ge 2$ for each $j=0,\dots,k-1$. Thus, considering \eqref{ineq log 2+x}  with $x=L_j(t)$, we get
\[
L_{j+1}(t)
=
\log\bigl(2+L_j(t)\bigr)
\le L_j(t),
\qquad \text{for }j=0,\dots,k-1.
\]
Thus, we conclude that, for $0<t<t_k^*$, we have
\begin{equation}\label{ineq 22}
L_k(t)\le L_{k-1}(t)\le\cdots\le L_1(t)\le L_0(t)
.
\end{equation}
Now, note that for $0<t\le 1$, we have the identity
\[
L_1(t)
=
\log\left(2+\frac1t\right)
=
\log(1/t)+\log(1+2t).
\]
Then, since $1+2t\le 3$, we  get
\[
L_1(t)
\le
\log(1/t)+\log 3
.
\]
Thus, if $0<t\le 1/e$, we have that $\log(1/t)\ge 1$, and we can write 
\[
\log 3\le \log 3\,\log(1/t).
\]
This implies that, 
\begin{equation}\label{ineq 33}
L_1(t)
\le
(1+\log 3)\log(1/t)
,
\end{equation}
for $0<t\le 1/ e$.  
Now,  we  set $t_0:=\min\left\{t_k^*,\frac1e\right\}$. 
Then, for $0<t<t_0$, combining \eqref{ineq 22} and \eqref{ineq 33}, we obtain
\[
\omega_k(t)^{-1}
=
L_k(t)
\le L_1(t)
\le
(1+\log 3)\log(1/t),
\]
and thus we conclude  that (H3) holds considering $C=1+\log 3$.
%which is independent of $k$.
%%%%%%%%%%%%%%%%%%
%%%%%%%%%%%%%%%%%%
\subsection{Examples 3 and 4}\label{subsec.4.3}
In the following examples we consider the following lemma to prove assumption (H3).
\begin{lemma}\label{forexamples}
If $\omega : (0, +\infty) \to (0, +\infty)$ satisfies $\lim_{t \to 0^+} \omega(t) = L$ for some $L \in (0, +\infty)$, then $\omega$   satisfies assumption (H3).
\end{lemma}

\begin{proof}
Since $L > 0$, we can write 
\[
\lim_{t \to 0^+} \omega(t)^{-1} = 1/L
.
\] 
Note that, there exist $t_1 > 0$ and $M > 0$ such that  
\[
\omega(t)^{-1} \leq M
\]
for $0 < t < t_1$ (taking, for instance, $M = \frac{1}{L} + 1$.
Now, since $\log(1/t) \to + \infty$ as $t \to 0^+$, there exists $t_2$ such that $\log(1/t) \geq M$ for $t < t_2$.
% (podemos tomar $t_2 = e^{-M}$ p.e.).  
Then, considering $t_0 = \min\{t_1, t_2 \}$, we obtain  
\[
\omega(t)^{-1} \leq M \leq \log(1/t)
\]  
in $(0, t_0)$. This concludes the proof.
\end{proof}
%%%
\subsubsection{Example 3}
We consider the weight $\omega(t) := 2+\arctan(t)$, for $t>0$. In the next we will verify (H1)--(H3).

First note that, the function $\arctan$ is strictly increasing on $\mathbb{R}$ (since $\arctan'(t)=\frac{1}{1+t^2}>0$ for all $t$). Hence, we have  $\omega(t)=2+\arctan(t)$ is strictly increasing on $(0,+\infty)$. On the other hand, for $t>0$,  $\arctan(t)>0$, thus we get $\omega(t)>2>0$, and then $\omega:(0, + \infty)\to(0, + \infty)$ is well defined.\\

Now, note that, since $\arctan(t)<\frac{\pi}{2}$ for all $t\in\mathbb{R}$, we have
\[
\omega(t) = 2+\arctan(t) < 2+\frac{\pi}{2} =: C_\omega, 
\qquad t>0
.
\]

Finally, we stress the fact that, by continuity of $\arctan$ at $t=0$, we can write 
\[
\lim_{t\to0^+}\omega(t) = 2+\arctan(0) = 2+0 = 2 \in (0,+\infty).
\]
Then, as this limit is finite and strictly positive, considering Lemma \eqref{forexamples},  we conclude that $\omega$ satisfies (H3). 
Hence $\omega(t)=2+\arctan(t)$ satisfies (H1)--(H3).  
 %%%%%%
\subsubsection{Example 4}
We consider now the weight $\omega(t) := 3-\dfrac{1}{1+t}$, for $t>0$.  In the next we will verify (H1)--(H3).\\

First, note that the function $t\mapsto \frac{1}{1+t}$ is strictly decreasing on $(0, + \infty)$ (its derivative is $-\frac{1}{(1+t)^2}<0$). Hence, we have that  $t\mapsto -\frac{1}{1+t}$ is an strictly increasing map, and therefore we get that  $\omega(t) = 3-\frac{1}{1+t}$ is strictly increasing on $(0, + \infty)$. \\

On the other hand, for $t>0$, we have $0<\frac{1}{1+t}<1$. Then, we can write
\[
2 = 3-1 < \omega(t) < 3-0 = 3.
\]
In particular we note that  $\omega(t)>0$, and thus $\omega:(0, + \infty)\to(0, + \infty)$ is well defined. 
Moreover,  $\omega(t)<3=:C_\omega$ for all $t>0$.\\

Finally, we emphasize that, by continuity of the map $t\mapsto \frac{1}{1+t}$ at $t=0$, we have 
\[
\lim_{t\to0^+}\omega(t) = 3-\frac{1}{1+0} = 3-1 = 2 \in (0, + \infty).
\]
Then, since this limit is finite and strictly positive, Lemma \ref{forexamples} applies with $L=2$, and thus we conclude that $\omega$ satisfies assumption (H3).

%%%%
\subsection{Example 5}\label{subsec.4new}
Our next example is based in the following lemma.
\begin{lemma}\label{forexamples} Let $\eta:(0,+\infty)\to(0,+\infty)$ be a function satisfying the following two conditions:
 \begin{itemize}
     \item $\lim_{t \to 0^+} \eta(t) = L$ for some $L \in (0, +\infty)$, and
     \item there exists $\delta\in(0,1)$ such that $\eta$ is non-decreasing and continuous in $(0,\delta]$.
 \end{itemize}
 Then the function $\omega:(0,+\infty)\to(0,+\infty)$, defined as
$$
\omega(t)=
\left\{
\begin{array}{ccl}
\eta(t)&\mbox{if}& t\leq\delta;\\
\eta(\delta)&\mbox{if}& t>\delta.\\
\end{array}
\right.
$$ 
satisfies (H1)-(H3).
 \end{lemma}
\begin{proof} Let $\lambda:(0,1)\to(0,+\infty)$ be defined as $\lambda(t)=-1/\log(t)$. We have
$\lambda(0+)=0<L=\omega(0+)$,  $\omega(1-)=\eta(\delta)<\infty=\lambda(1-)$, and, $\lambda$ and $\omega$ are continuous in  $(0,1)$. Hence there exists $t_0\in(0,1)$ such that  $\lambda(t)\leq \omega(t)$, for all $t\in(0,t_0)$. Consequently, $\omega$ is non-decreasing, continuous, and bounded by $\eta(\delta)$. Moreover
$$
\frac{1}{\omega(t)}\leq 
\frac{1}{\lambda(t)}=\log(1/t),\quad\mbox{for all } t\in(0,t_0).
$$
This completes the proof.
\end{proof}

\subsubsection{Example 5} Consider $\eta:(0,+\infty)\to(0,+\infty)$ defined as $\eta(t)=2-\exp(-t)$. Then,  the corresponding $\omega$, with $\delta=\frac 1 2$,  satisfies (H1)-(H3).

{\color{black}
\subsection{Example 6: a dyadic weight}\label{subsec.4.4}

To begin, given $t\in(0,1]$ we define  
\[
N(t):=\max\{n\in\mathbb Z_{\ge0}\ :\ 2^{-n}\ge t\}.
\]
Note that the the maximum above exists and is
unique, since the set on the right-hand side of the expression above is nonempty ( $n=0$ is contained there as $t\le1$) and
bounded above (as $2^{-n}\to0$ as $n\to\infty$). 
In particular, we stress the fact that $N(t)$ is the unique integer $n\ge0$ such that
\begin{equation}\label{eq:dyadic-interval}
t\in\big(2^{-(n+1)},2^{-n}\big].
\end{equation}
With this at hand, we define the weight $\omega_{\text{dyadic}}:(0,+\infty)\to(0,+\infty)$ as 
\[
\omega_{\text{dyadic}} (t):=
\begin{cases}
1, & t\ge1,\\[4pt]
\dfrac{1}{n+2}+\Big(\dfrac1{n+1}-\dfrac1{n+2}\Big)\dfrac{t-2^{-(n+1)}}{2^{-n}-2^{-(n+1)}}, & t\in\big(2^{-(n+1)},2^{-n}\big],\ n=N(t).
\end{cases}
\]
We claim that $\omega_{\text{dyadic}} $ in fact satisfies (H1)--(H3). \\
 
To begin, note that  $\omega_{\text{dyadic}} $ is continuous on $(0,+\infty)$.
 In fact,  on the interior of each interval $(2^{-(n+1)},2^{-n})$, $\omega_{\text{dyadic}} $ is affine in $t$, hence continuous. Now, by  construction we know t that  on the interval $ (2^{-(n+1)},2^{-n}]$ we have $\omega_{\text{dyadic}} (2^{-n})=\frac{1}{n+1}$, 
while the adjacent interval on the right, $(2^{-n},2^{-(n-1)}]$,
also gives the value $\frac{1}{n+1}$ at its left endpoint. Thus, we conclude that there are not discontinuities at $t=2^{-n}$. On the other hand, at $t=1$, we have  $\omega_{\text{dyadic}} (1)=1$, which agrees with the constant definition $\omega_{\text{dyadic}} (t)=1$ for $t\ge 1$.  \\ 

To continue, note that on each interval $I_n=(2^{-(n+1)},2^{-n}]$, the weight $\omega_{\text{dyadic}} $ is affine, with slope
$\frac{\frac{1}{n+1}-\frac{1}{n+2}}
{2^{-n}-2^{-(n+1)}}>0$
.
Hence $\omega_{\text{dyadic}} $ is strictly increasing on every $I_n$.
Now, considering that for each $n\ge 0$ we have
\[
\omega_{\text{dyadic}} (2^{-n})=\frac{1}{n+1}, 
\]
so consecutive intervals meet at the same value at every dyadic endpoint. 
Thus, we conclude that $\omega_{\text{dyadic}} $ is nondecreasing on $(0,1]$.
Since $\omega_{\text{dyadic}} (t)=1$ for $t\ge 1$ and $\omega_{\text{dyadic}} (1)=1$, it follows that $\omega_{\text{dyadic}} $ is nondecreasing on $(0,\infty)$. Thus, (H1) holds.\\

On the other hand, for $t\ge1$,  we know that $\omega_{\text{dyadic}} (t)=1$ by definition, and for $t\in(0,1)$, with $n=N(t)\ge0$, the monotonicity on the interval (proved above)  yields
\[
\omega_{\text{dyadic}} (t)\le\omega_{\text{dyadic}} (2^{-n})=\frac1{n+1}\le1.
\]
Hence $\omega_{\text{dyadic}} (t)\le1$ for every $t>0$, and (H2) holds with $C_\omega=1$.\\
 
 For proving (H3), we will use the following identity
\begin{equation}\label{eq:dyadic-count}
N(t)=\Big\lfloor \log_2\tfrac1t \Big\rfloor,\qquad t\in(0,1],
\end{equation}
which is a consequence of  \eqref{eq:dyadic-interval}  after applying $\log_2$ to both sides of the inequality. 
To continue, let  $0<t<\frac12$ and set $n:=N(t)$. 
Thus, considering that $t< 2^{-1}$ and the definition of $N(t)$ we get that  $n\ge1$.
Then, using the monotonicity of $\omega_{\text{dyadic}} $ on the interval with index $n$ (established in the proof of (H1)), and as $t>2^{-(n+1)}$, we can write 
\[
\omega_{\text{dyadic}} (t)\ \ge\ \lim_{s\to(2^{-(n+1)})^+}\omega_{\text{dyadic}} (s)=\frac1{n+2},
\]
and hence, we get  
\begin{equation}\label{eq:H3-a}
\omega_{\text{dyadic}} (t)^{-1}\le n+2.
\end{equation}
On the other hand, by \eqref{eq:dyadic-interval} we know that  $t\le2^{-n}$, therefore $1/t\ge2^n$
and we obtain 
\begin{equation}\label{eq:H3-b}
\log(1/t)\ge n\log2.
\end{equation}
Now, since the map $n\mapsto\frac{n+2}{n}=1+\frac2n$ is decreasing for $n\ge1$.
Hence its maximum is achieved at $n=1$ and is equals to $3$, which yields   $n+2\le3n$ for all $n\ge1$. 
Then, gathering this with \eqref{eq:H3-a} and \eqref{eq:H3-b}, we get
\[
\omega_{\text{dyadic}} (t)^{-1}\ \le\ n+2\ \le\ 3n\ \le\ \frac{3}{\log2}\,\log(1/t),
\]
for all $0<t<\frac12$. Thus, we conclude that  (H3) holds with $C=\frac3{\log2}$ and $t_0=\frac12$.

\begin{remark}
Note that, $\omega_{\text{dyadic}} (t)\to0$ as $t\to0^+$. In fact, since $N(t)\to+\infty$ as $t\to0^+$ and as 
\[
\omega_{\text{dyadic}} (t)\in\Big[\frac1{n+2},\frac1{n+1}\Big],\qquad n=N(t),
\]
and both extremes of the interval tend to $0$ as $n\to\infty$, the squeeze theorem yields the claim.
\end{remark}
}

%%%%%%%%%%%%%%% 
 \section{Proofs of Lemma \ref{lem.phiq}}\label{proof.main.lemma}

For $ t>0$ consider
\[
h(t):=t^{q-1}\omega(t)
.
\]
Since $\omega$ is continuous and positive on $(0, + \infty)$, we conclude that $h$ is continuous and positive on $(0, + \infty)$. On the other hand, as we assume that $\omega$ is bounded (H2), we get
\[
0<h(t)\le C_\omega t^{q-1},\qquad t>0.
\]
As a consequence, and since $q>1$, we have
$C_\omega t^{q-1}\longrightarrow 0$  as $t \to 0^+$, and we can write 
\[
h(t)\to 0
\qquad\text{as }t\to0^+.
\]
Thus, defining
$h(0):=0$, we conclude that $h\in C([0, + \infty))$.
To continue, note that, for every $T>0$, we have
\[
\phi_q(T)=
\int_0^T h(\xi)\,d\xi
\le
 C_\omega\int_0^T \xi^{q-1}\,d\xi
=
\frac{C_\omega}{q}T^q< + \infty
.
\]
Thus, we conclude that $\phi_q$ is well defined on $[0, + \infty)$.
On the other hand, considering that $h$ is continuous on $[0, + \infty)$,  with the help of the fundamental theorem of calculus we deduce that 
\[
\phi_q\in C^1([0, + \infty))
\]
and, for $t\ge0$ we have
\[
\phi_q'(t)=h(t)=t^{q-1}\omega(t)
.
\]
where at $t=0$ both sides of the identity are understood as $0$. A first consequence of this, we get
\[
\phi_q(0)=0,
\]
and for each $t>0$,
\[
\phi_q'(t)=t^{q-1}\omega(t)>0
.
\]
Note in particular that $\phi_q$ is strictly increasing on $[0, + \infty)$.
To continue, we will focus on proving convexity. Since $q>1$, the map $t\mapsto t^{q-1}$ is nondecreasing on $[0, + \infty)$, while $\omega$ is nondecreasing by (H1). Then, since both factors are nonnegative, their product
$h(t)=t^{q-1}\omega(t) $ is also nondecreasing. Thus, we conclude $\phi_q'=h$ is nondecreasing, and hence $\phi_q$ is convex on $[0, + \infty)$.\\

To continue, note that since $\omega$ is nondecreasing (by (H1)), for every $\xi\ge1$, we have $\omega(\xi)\ge\omega(1)>0. $ Therefore, for $t\ge1$, we can write
\[
\begin{aligned}
\phi_q(t)
&=\int_0^t \xi^{q-1}\omega(\xi)\,d\xi\\
&\ge \int_1^t \xi^{q-1}\omega(\xi)\,d\xi\\
&\ge \omega(1)\int_1^t \xi^{q-1}\,d\xi\\
&=\frac{\omega(1)}{q}(t^q-1).
\end{aligned}
\]
Since $q>1$, we have $
\frac{\omega(1)}{q}(t^q-1)\longrightarrow+\infty$ as $t \to + \infty$. Hence, we conclude
\[
\phi_q(t)\longrightarrow+\infty
\qquad\text{as }t\to\infty.
\]

Finally, consider $\tau>0$. Since $\omega$ is nondecreasing, we know that $\omega(\xi)\le\omega(\tau)$  for  $ 0<\xi\le\tau$. As a consequence, we can write
\[
\begin{aligned}
\phi_q(\tau)
&=
\int_0^\tau \xi^{q-1}\omega(\xi)\,d\xi\\
&
\le
\omega(\tau)\int_0^\tau \xi^{q-1}\,d\xi\\
&=
\frac{\tau^q\omega(\tau)}{q}
.
\end{aligned}
\]
For $\tau=0$, both sides of the expression above are equal to $0$. Thus, we get 
\[
\phi_q(\tau)\le\frac{\tau^q\omega(\tau)}{q},
\qquad \tau\ge0.
\]

It remains to prove the assertion concerning the inverse. We have shown that $\phi_q$ is continuous and strictly increasing on $[0,+ \infty)$, with $
\phi_q(0)=0$  and $ \displaystyle
\lim_{t\to\infty}\phi_q(t)=+\infty$. Therefore, we conclude that 
\[
\phi_q([0,+ \infty))=[0,+ \infty),
\]
and then the map $\phi_q:[0,+ \infty)\to[0,+ \infty) $ 
is a continuous strictly increasing bijection. Then, its inverse $\phi_q^{-1}:[0,+ \infty)\to[0,+ \infty)
$ is consequently continuous and strictly increasing. With this, we conclude 
\[
\phi_q^{-1}(0)=0.
\]
With this, we conclude the proof of the lemma.\\

\paragraph{\bf Acknowledgements} The second author warmly thanks  M\'onica Candia-Reyes for her valuable advice and insightful comments.
 
\paragraph{\bf Datasets} Data sharing does not apply to this article as no datasets were generated or analyzed during the
current study.
 
\paragraph{\bf  Conflict of interest} In addition, the author declares no conflict of interest and confirms being the only contributors to this paper.
 
%------------------
%------ end of document 
%------------------
%\newpage


\begin{thebibliography}{99}
%\addcontentsline{toc}{section}{References}
\bibitem{chae14}
{\sc D.~Chae}, {\em Liouville-type theorems for the forced {E}uler equations
  and the {N}avier-{S}tokes equations}, Comm. Math. Phys., 326 (2014),
  pp.~37--48.
%
\bibitem{ChaeWolf}
{\sc D.~Chae and J.~Wolf}, {\em On {L}iouville type theorems for the steady
  {N}avier-{S}tokes equations in {$\mathbb{R}^3$}}, J. Differential Equations,
  261 (2016), pp.~5541--5560.
%
\bibitem{galdi2011introduction}
{\sc G.~Galdi}, {\em An introduction to the mathematical theory of the
{N}avier-{S}tokes equations: Steady-state problems}, Springer Science \&
Business Media, 2011.
%
\bibitem{Kozonoetal}
{\sc H.~Kozono, Y.~Terasawa, and Y.~Wakasugi}, {\em A remark on
  {L}iouville-type theorems for the stationary {N}avier-{S}tokes equations in
  three space dimensions}, J. Funct. Anal., 272 (2017), pp.~804--818.
%
\bibitem{lemarie2016navier}
{\sc P.~G. Lemari{\'e}-Rieusset},
{\em The {N}avier-{S}tokes problem in the 21st century}, CRC press, 2016.
%
%
\bibitem{Lerner26}
{\sc N.~
Lerner},   {\em Wiener Algebras Methods for Liouville Theorems on the Stationary Navier-Stokes System}, arXiv preprint arXiv:2601.13916  (2026).
%
\bibitem{Ser2016}  
G. \textsc{Seregin}, \emph{A Liouville type theorem for steady-state Navier-Stokes equations}.  
J. {\'E}.D.P., Expos{\'e} no IX, (2016).
%
\bibitem{Seregin16}
{\sc G.~Seregin}, {\em Liouville type theorem for stationary {N}avier-{S}tokes
  equations}, Nonlinearity, 29 (2016), pp.~2191--2195.
%
\bibitem{Sereginetwang}
{\sc G.~Seregin and W.~Wang}, {\em Sufficient conditions on {L}iouville type
  theorems for the 3{D} steady {N}avier-{S}tokes equations}, Algebra i Analiz,
  31 (2019), pp.~269--278.
%
\bibitem{V2026}
 {\sc G. Vergara-Hermosilla}, \emph{Liouville theorems above the critical $9/2$ Threshold for stationary Navier-Stokes Equations}, arXiv preprint arXiv:2604.06527, 2026.
\end{thebibliography}
 \end{document}